\documentclass[12pt]{amsart}
\usepackage{amsmath,amssymb, xypic,verbatim,amscd,color}
\usepackage{graphicx}
\usepackage{colordvi}

\numberwithin{equation}{section}

\newcommand\frg{\mathfrak{g}}
\newcommand\frh{\mathfrak{h}}

\newcommand\mm{\mathcal{M}}

\newcommand\home{\operatorname{Hom}}

\newcommand\pone{\Bbb{P}^1}
\newcommand\neww{\operatorname{Anc}}

\newcommand\Id{1}

\newcommand\mv{\mathcal{V}}
\newcommand\tensor{\otimes}

\newcommand\mh{\mathcal{H}}

\newcommand\mc{\mathcal{C}}

\newcommand\Res{\operatorname{Res}}

\newcommand{\leto}[1]{\stackrel{#1}{\to}}

\newtheorem{theorem}{Theorem}[section]
\newtheorem{remark}[theorem]{ Remark}

\newtheorem{proposition}[theorem]{Proposition}

\newtheorem{lemma}[theorem]{Lemma}

\newtheorem{definition}[theorem]{Definition}
\newtheorem{defi}[theorem]{Definition}

\begin{document}
\title[Unitarizability of the KZ connection]{Unitarity of the KZ/Hitchin connection on conformal blocks in genus $0$ for arbitrary Lie algebras}

\author{Prakash Belkale}
\thanks{Partially supported by NSF grant  DMS-0901249.}

\address{Department of Mathematics\\ UNC-Chapel Hill\\ CB \#3250, Phillips Hall
\\ Chapel Hill, NC 27599}
\email{belkale@email.unc.edu}
\begin{abstract} We prove that the vector bundles of conformal blocks, on suitable moduli spaces of genus zero curves with marked points, for arbitrary simple Lie algebras and arbitrary integral levels, carry unitary metrics of geometric origin  which  are preserved by the Knizhnik-Zamolodchikov/Hitchin connection (as conjectured in ~\cite{FGK}). Our proof builds upon the work of Ramadas ~\cite{TRR} who proved this unitarity statement in the case of the Lie algebra $\mathfrak{sl}_2$ (and genus $0$).
\end{abstract}

\maketitle

\section{Introduction}
Consider  a finite dimensional simple Lie algebra $\frg$, a non-negative integer $k$ called the level and a $N$-tuple $\vec{\lambda}=(\lambda_1,\dots,\lambda_N)$ of
dominant weights of $\frg$ of level $k$. The mathematical theory of Tsuchiya-Kanie ~\cite{TK} and Tsuchiya-Ueno-Yamada ~\cite{TUY}, associates to this data a vector bundle $\mathcal{V}=\mathcal{V}_{\vec{\lambda},k}$ on $\overline{\mathfrak{M}}_{g,N}$, the moduli
stack of stable $N$-pointed curves of genus $g$.

On the open part $\mathfrak{M}_{g,N}$ of smooth pointed curves, $\mv$ carries a flat projective connection $\nabla$, which  is the restriction of a suitable Knizhnik-Zamolodchikov (KZ) connection when $g=0$. The WZW connection ~\cite{TUY} generalizes the KZ connection to all genera.

The fibers of $\mv$ on $\mathfrak{M}_{g,N}$  can also be described in terms of sections of natural line bundles on suitable moduli stacks of parabolic principal bundles on $N$-pointed curves of genus $g$. These
sections generalize classical theta functions, and are hence called non-abelian or generalized theta functions (see the survey ~\cite{sorger}).
The connection on $\mv$ was described from the above algebro-geometric point of view by Hitchin.

A basic conjecture in the subject, with origins in physics, is that $\mv$ carries  a projective unitary metric  which is preserved (projectively) by the connection $\nabla$. This conjecture has been proved for all genera including genus $0$, by the combined work of Kirillov and Wenzl via topological field theory and the theory of quantum groups  ~\cite{KL,Ki1,Ki2,W, Ki3}. The construction of this metric, which is expected to be unique after some additional conditions are imposed, is however not explicit.

In the 90's Gawedzki  and collaborators ~\cite{Gawedzki1,FGK} proposed a conjectural construction of the unitary metric via integration of the  Schechtman-Varchenko forms ~\cite{SV} (see ~\cite{SV,BF,MTV,V} and the references therein for various applications of these beautiful forms).  Recently the case $\frg=\mathfrak{sl}_2$ and genus $0$ of Gawedzki's proposal was rigorously proved by Ramadas ~\cite{TRR}. In this paper, following Ramadas' general strategy, we prove the (geometric) unitarity conjecture for arbitrary simple Lie algebras $\frg$ in genus $0$. As in Ramadas' work, the unitary metric is obtained by realizing
the bundle of conformal blocks inside a Gauss-Manin system of cohomology of  smooth projective  varieties.
\begin{theorem}\label{introtheorem}
The KZ/Hitchin connection on  bundles of conformal blocks over configuration spaces of distinct points on $\Bbb{A}^1$ is unitary, with the unitary metric of geometric origin, for any  simple Lie algebra $\frg$ and any integral level $k$.
\end{theorem}

An algebro-geometric strategy for carrying out  Gawedzki's proposal was given in ~\cite{TRR}. According to Ramadas, one should  first find derivatives of theta functions on Harder-Narasimhan strata, which we take  as a suggestion to look at correlation functions. The main steps in Ramadas' proof ~\cite{TRR} can be described as follows (here $g=0$ and  $\frg=\mathfrak{sl}_2$).
\begin{enumerate}
\item Fix a curve $C$ of genus $g$ with $N$ marked points $z_1,\dots, z_N$. The fiber $V$ of $\mv$ over this marked curve is isomorphic to the space of sections of a natural line bundle on a suitable moduli stack of parabolic principal bundles on $C$ (with parabolic structure  at the given marked points). Any global section of this
    line bundle necessarily vanishes on any Harder-Narasimhan stratum (which corresponds to  non-semistable parabolic bundles). Ramadas'
    first step is to consider a ``Harder-Narasimhan trace'', which is a suitable higher order derivative,  at such strata. These considerations lead him to a map from $V$ to
    the space of top degree differential forms on a affine variety $X$ that depends upon the marked curve (actually $X$ is an open subset of $C^M$ for a suitable $M$).
\item Ramadas proves (geometrically) a key vanishing statement  for such differential forms  on partial diagonals.
He then lifts these differential forms to a finite cover $Y$ of $X$, multiplies them by a ``master function'' and shows, using the vanishing statement, that the resulting differential forms extend to any smooth compactification $\overline{Y}$  of $Y$.
\item We hence obtain  an injective  map $V\to H^{\dim{Y}}(\overline{Y},\Bbb{C})$. Ramadas then proves that this map is flat for the connections as the marked curve varies (and where we consider the Gauss-Manin connection on the vector bundle formed by the spaces  $H^{\dim{Y}}(\overline{Y},\Bbb{C})$). To prove the above flatness assertion, it suffices to prove that the composite $V\to H^{\dim{Y}}(Y,\Bbb{C})$ is flat, which is  verified by an explicit calculation.
\item We are now in  an curious situation, with a flat map from $\mv$ to a Gauss-Manin system (of cohomology of smooth projective varieties)  whose image  is contained in the space of  top-degree algebraic forms. It follows that the canonical polarization on the Gauss-Manin system restricts to give the desired metric on $\mv$.
\end{enumerate}
\subsection{An enumeration of our main steps:}
We modify the first step in Ramadas' proof by working in the language of conformal blocks ~\cite{TK,TUY,Ueno}. The differential form
that we consider is a suitable correlation function. It plays a similar role as Ramadas' Harder-Narasimhan trace (I would like to acknowledge this conceptual starting point given by Ramadas). Roughly speaking, it corresponds to modifying the corresponding $G$-bundle (in the simplest Harder-Narasimhan stratum) around a finite set of additional points, and then taking a suitable mixed partial derivative of the theta functions in the direction of these changes (the underlying principal bundle is actually trivial in this stratum). The modifications are controlled by the choice of  simple positive roots of the Lie algebra for each of the additional points.

The correlation function is a differential form in the additional variables, with remarkable confluence properties as some of the additional points coalesce at the parabolic points, at a fixed point ($\infty$ for us) and at a moving point. The works  ~\cite{TK,TUY} give local expressions for these correlation functions, allowing a bridge to the  representation theory of Kac-Moody algebras (see Proposition ~\ref{richardthomas}).

The master function is the function discovered by Schechtman-Varchenko ~\cite{SV} (see Section ~\ref{cottard}). We prove the extension property of the differential form (the master function multiplied with the correlation function) to smooth compactifications by using the confluence properties alluded to in the previous paragraph. In particular, we prove more general forms of ``vanishing along partial diagonals''  by employing these techniques.

The crucial flatness assertion is proved by using results of Schechtman-Varchenko ~\cite{SV}. The new ingredient is an interpretation of the correlation function in terms of the Schechtman-Varchenko forms, see equation  ~\eqref{interpretation} (this uses ideas that I learned from ~\cite{ATY}).

Exactly as in Ramadas' paper, the last step (the fourth) involves applications of mixed Hodge theory.
\subsection{Acknowledgements}\label{papers}  This paper builds upon the beautiful recent work of T.R. Ramadas ~\cite{TRR}.  In addition, the papers ~\cite{loo,Var}  were helpful in our work (in particular, we use ideas from these papers).

I am deeply indebted to  Najmuddin Fakhruddin for an extensive  communication on the subject, his comments, clarifications and corrections, constant encouragement, and for interest in this work. I thank Madhav Nori for bringing Ramadas' work to my attention (in 2004), and for many conversations over the years on related questions. I thank A.  Varchenko for pointing out many  references (in particular, the works of Kirillov and Wenzl), and for his comments on an earlier version. I thank A. Kirillov, S. Kumar, E. Looijenga  and H. Wenzl for useful discussions and correspondence.

\section{Basic definitions in  the theory of conformal blocks}
We recall some definitions from Ueno's book ~\cite{Ueno}, which we will use as a reference for the theory of conformal blocks. Let $\frg$ be a simple Lie algebra over $\Bbb{C}$. Fix the data of a  Cartan decomposition of $\frg$:
$$\frg=\frh\oplus\sum_{\alpha\in\Delta}\frg_{\alpha}$$
where $\Delta\subseteq \frh^*$ the set of roots is decomposed into a union $\Delta_{+}\cup \Delta_{-}$ of positive and negative roots. The set of simple (positive) roots is denoted by $R$.

A Killing form $(\ , \ )$ on $\frg$ induces one on $\frh$ and $\frh^*$. Normalize the Killing form by requiring that  $(\theta,\theta) = 2$.
\subsection{Affine Lie algebras}
The affine Lie algebra $\hat{\frg}$ is defined to be
$$\hat{\frg}=\frg\tensor \Bbb{C}((\xi))\oplus \Bbb{C}c$$
where $c$ is an element in the center of $\hat{\frg}$ and the Lie algebra structure is defined by
$$[X\tensor f(\xi),Y\tensor g(\xi)]= [X,Y]\tensor f(\xi)g(\xi) + c(X,Y)\Res_{\xi=0}(gdf)$$
where $f,g\in \Bbb{C}((\xi))$ and $X,Y\in \frg$.

Introduce notation  for the Lie subalgebras
 $$\hat{\frg}_{+}=\frg\tensor \Bbb{C}[[\xi]]\xi,\ \hat{\frg}_{-}=\frg\tensor \Bbb{C}[\xi^{-1}]\xi^{-1}$$ so that we have a vector space decomposition
 $$\hat{\frg}=\hat{\frg}_{+}\oplus \frg\oplus \Bbb{C}c\oplus\hat{\frg}_{-}.$$
Let $$X(n)= X\tensor \xi^n, \ X=X(0)=X\tensor 1,\ X\in \frg, n\in \Bbb{Z}.$$

\subsection{Representation theory of affine Lie-algebras}
 Recall that finite dimensional irreducible representations of $\frg$ are parameterized by the set of dominant integral weights $P_+$  considered a subset of $\frh^*$. To $\lambda\in P_+$, the corresponding irreducible representation $V_{\lambda}$ contains a non-zero vector $v\in V_{\lambda}$ (the highest weight vector) such that
 $$Hv=\lambda(H)v, H\in \frh$$ $$X_{\alpha}v=0, X_{\alpha}\in \frg_{\alpha}, \forall \alpha\in \Delta_{+}.$$

 We will fix a level $k$ in the sequel. Let $P_k$ denote the set of dominant integral weights of level $k$. More precisely
 $$P_k=\{\lambda\in P_+\mid (\lambda,\theta)\leq k\}$$
where  $\theta$ is the highest (longest positive) root.

For each $\lambda\in P_k$ there is a unique irreducible representation $\mathcal{H}_{\lambda}$ of $\hat{\frg}$ which satisfies the following properties
\begin{enumerate}
\item $V_{\lambda}=\{|v\rangle\mid \hat{\frg}_+|v\rangle =0\}$.
\item The central element $c$ of $\hat{\frg}$ acts on $\mh_{\lambda}$ by multiplication by $k$.
\item Let $|\lambda\rangle$ denote a highest weight vector in $V_{\lambda}$ and $\theta$ the highest root, then
    $$X_{\theta}(-1)^{k-(\theta,\lambda)+1}|\lambda\rangle=0.$$
    where $X_{\theta}$ is a non-zero element in the root space $\frg_{\theta}$. In fact $\mh_{\lambda}$ is generated by $V_{\lambda}$ over $\hat{\frg}$ with the above fundamental relation.
\end{enumerate}
The representation $\mathcal{H}_{\lambda}$ when $\lambda=0$ (still at level $k$) is called the vacuum representation at level $k$.

\subsection{Conformal blocks}
We will work with conformal blocks on marked curves in $\mathfrak{M}_{0,N}$ (that is, smooth curves of genus $0$ with $N$ marked points). But we will state the definitions in greater generality.

To define conformal blocks  we will fix a stable
 $N$-pointed curve of genus $g$ with formal neighborhoods $\mathfrak{X}=(C; P_1,\dots,P_N,\eta_1,\dots,\eta_N)$. Here we require:
 \begin{enumerate}
 \item $C$ is smooth at the distinct points $P_1,\dots,P_N$.
\item $C-\{P_1,\dots,P_N\}$ is affine.
\item  A stability condition (finiteness of automorphisms of the pointed curve).
\item Isomorphisms $\eta_i: \hat{\mathcal{O}}_{C,P_i}\leto{\sim}
\Bbb{C}[[\xi_i]],\ i=1,\dots, N$.
\end{enumerate}
Let $\mathfrak{X}$ and be as above, and choose
$\vec{\lambda}=(\lambda_1,\dots,\lambda_N)\in P_k^N$. There are a
number of definitions relevant to the situation: Let
$$\hat{\frg}_N=\bigoplus_{i=1}^N \frg\tensor_{\Bbb{C}}
\Bbb{C}((\xi_i))\oplus \Bbb{C}c.$$ be the Lie algebra with $c$ a
central element and the Lie bracket given by
$$[\sum_{i=1}^N X_i\tensor f_i,\sum_{i=1}^N Y_i\tensor g_i]= \sum_{i=1}^N [X_i,Y_i]\tensor f_ig_i + c
\sum_{i=1}^N (X_i,Y_i) \operatorname{Res_{P_i}}(g_i df_i).$$

Let
\begin{equation}\label{sato}
\frg(\mathfrak{X})=\frg\tensor_{\Bbb{C}}\Gamma(C-\{P_1,\dots,P_N\},\mathcal{O})\hookrightarrow
\hat{\frg}_N.
\end{equation}
Let $\vec{\lambda}$ be as
above. Set
$$\mh_{\vec{\lambda}}=\mh_{\lambda_1}\tensor\dots\tensor\mh_{\lambda_N}.$$

For a given $X\in \frg$ and $f\in\Bbb{C}((\xi_i))$, define
$\rho_i(X\tensor f)$ an endomorphism of $\mh_{\vec{\lambda}}$ by
$$\rho_i(X\tensor f)| v_1\rangle\tensor\dots\tensor |v_N\rangle=|
v_1\rangle\tensor\dots \tensor(X\tensor f|v_i\rangle) \tensor\dots\tensor |v_N\rangle$$
where  $| v_i\rangle\in\mh_{\lambda_i}$ for each $i$.

We can now define the action of $\hat{\frg}_N$ on
$\mh_{\vec{\lambda}}$ by
$$(X_1\tensor f_1,\dots, X_N\tensor f_N)| v_1\rangle\tensor\dots\tensor
|v_N\rangle = \sum_{i=1}^N \rho_i(X_i\tensor f_i)| v_1\rangle\tensor\dots\tensor
|v_N\rangle.$$
\begin{defi}
Define the space of conformal blocks
$$V^{\dagger}_{\vec{\lambda}}(\mathfrak{X})=
\home_{\Bbb{C}}(\mh_{\vec{\lambda}}/\frg(\mathfrak{X})\mh_{\vec{\lambda}},\Bbb{C})$$
Define the set of dual conformal blocks,
$V_{\vec{\lambda}}(\mathfrak{X})=\mh_{\vec{\lambda}}/\frg(\mathfrak{X})\mh_{\vec{\lambda}}$.
These are both finite dimensional $\Bbb{C}$-vector spaces which can  defined on families (and commute with base change) ~\cite{Ueno}.
\end{defi}

Following Dirac's bra-ket conventions, elements of
$V^{\dagger}_{\vec{\lambda}}(\mathfrak{X})$ (or $\mh^*_{\vec{\lambda}}$)  are
frequently denoted by $\langle\Psi|$ and those of
$V_{\vec{\lambda}}(\mathfrak{X})$ (or of $\mh_{\vec{\lambda}}$) by
$|\Phi\rangle$ and the pairing by $\langle\Psi|\Phi\rangle$.

\begin{remark}
Let $\langle\Psi|\in V^{\dagger}_{\vec{\lambda}}$, $|\Phi\rangle\in \mh_{\vec{\lambda}}$, $X\in \frg$, and $f\in\Gamma(C-\{P_1,\dots,P_N\},\mathcal{O})$, then the following ``gauge condition'' holds (using the embedding ~\eqref{sato})
$$\langle\Psi|X\tensor f|\Phi\rangle=0.$$
\end{remark}
\subsection{Propagation of vacua}
Add a new point $P_{N+1}$ together with the vacuum representation
$V_0$ of level $k$, at $P_{N+1}$. Also fix a formal neighborhood at
$P_{N+1}$. We therefore have a new pointed curve $\mathfrak{X}'$,
and an extended
$\vec{\lambda}'=(\lambda_1,\dots,\lambda_N,\lambda_{N+1}=0)$. The
propagation of vacuum gives an  isomorphism
$$V^{\dagger}_{\vec{\lambda}'}(\mathfrak{X}')\leto{\sim}V^{\dagger}_{\vec{\lambda}}(\mathfrak{X}),
\ \langle\widehat{{\Psi}}|\mapsto \langle{\Psi}|$$ with the key formula
$$\langle\widehat{{\Psi}}|(|\Phi\rangle\tensor|0\rangle)= \langle{\Psi}|\Phi\rangle.$$

\subsection{Correlation functions}\label{corro}
 Suppose $\mathfrak{X}\in \mathfrak{M}_{g,N}$. Let $\langle{\Psi}|\in V^{\dagger}_{\vec{\lambda}}(\mathfrak{X})$, $Q_1,\dots,Q_M\in C-\{P_1,\dots,P_N\}$,
 $|\Phi\rangle\in \mh_{\vec{\lambda}}$, $Q_1,\dots,Q_M\in C-\{P_1,\dots,P_N\}$, $Q_i\neq Q_j, i<j$ and corresponding elements
 $X_1,\dots,X_M\in \frg$. There is a very important differential called a correlation function
$$\Omega=\langle{\Psi}|X_1(Q_1)X_2(Q_2)\dots X_M(Q_M)|\Phi\rangle\in \bigotimes_{i=1}^M \Omega^1_{C,Q_i}.$$

Here $\Omega^1_{C}$ is the vector bundle of holomorphic one-forms on $C$.
One way to define $\Omega$  is via propagation by vacua: add points $Q_1,\dots Q_M$ with formal coordinates $\psi_1,\dots,\psi_M$ and consider the elements $X_a(-1)|0\rangle$ in the vacuum representation at those points. Then
$$\Omega=\langle\widehat{{\Psi}}|X_1(-1)|0\rangle\tensor X_2(-1)|0 \rangle\dots X_M(-1)|0\rangle\tensor |\Phi\rangle d\psi_1\dots d{\psi_M}.$$
The differential form $\Omega$  is independent of the chosen coordinates (see page 70 of ~\cite{Ueno} for more details).

\section{Formulation of the extension theorem}\label{marcel}
We will henceforth consider the case $C=\pone$, with a chosen $\infty$ and a coordinate $z$ on $\Bbb{A}^1=\pone-\{\infty\}$.
Consider distinct points $P_1,\dots,P_N\in \Bbb{A}^1\subset \pone$ with $z$-coordinates $z_1,\dots,z_N$ respectively. The standard coordinate $z$ endows each $P_i$ with a formal coordinate. Let $\mathfrak{X}$ be the resulting $N$-pointed curve with formal coordinates.

\begin{definition}\label{gold}
For every positive root $\delta$, make a choice of a non-zero element $f_{\delta}$
in $\frg_{-\delta}$.
\end{definition}

Assume that we are given  $\lambda_1,\dots,\lambda_N\in P_k$, such
that $\mu=\sum_{i=1}^{N} \lambda_i$ is in the root lattice. Write
$\mu=\sum n_p \alpha_p$, where $\alpha_p$ are the simple positive roots. It is easy to see that each $n_p$ is non-negative
(for example, by evaluating both sides on $H_{\alpha_p}$).

Let $|\vec{\lambda}\rangle=|{\lambda}_1\rangle\tensor\dots\tensor
|{\lambda}_N\rangle$ be the product of the corresponding highest weight vectors.
 Now consider and fix a  map $\beta:[M]=\{1,\dots,M\} \to \Delta_+$, so that
$\mu=\sum_{a=1}^M \beta(a)$. The maximum value of $M$ is $M=\sum n_p$ and this is the main case. In such cases each $\beta(a)$ is necessarily a simple positive root. However, in various arguments, we will need to consider smaller values of $M$.

Introduce variables $t_1,\dots,t_M$ considered points on $\pone-\{\infty, P_1,\dots, P_N\}$
Consider, for every $\langle\Psi|\in
V^{\dagger}_{\vec{\lambda}}(\mathfrak{X})$, the correlation function
$$\Omega=\Omega_{\beta}(\langle\Psi|) = \langle\Psi|f_{\beta(1)}(t_1)f_{\beta(2)}(t_2)\dots f_{\beta(M)}(t_M)|\vec{\lambda}\rangle.$$
We will use the simplified notation $\langle \Psi| f_{\beta}(\vec{t})|\vec{\lambda}\rangle$
for the right hand-side of the above expression. Note that $\Omega$ has at most poles of the first order along the sum $D$ of the  divisors
\begin{enumerate}
\item[(a)] For  $1\leq a<b\leq M$ the divisor  given by the closure of $t_a=t_b$.
\item[(b)] For $a\in [M]$ and $i\in [N]$, the divisor $t_a= z_i$.
\end{enumerate}
\begin{remark}
The correlation function  $\Omega$ is regular at the generic point of the divisor $t_a=\infty$ for any $a$.
\end{remark}
 It is known that in the genus $0$ situation, conformal blocks embed in the $\frg$-invariants in the dual of the tensor product $(\tensor V_{\lambda_j})^*$ (see ~\cite{Ueno}, Proposition 6.1). The following is an amplification:

\begin{lemma}\label{simpsons}
The map $V^{\dagger}_{\vec{\lambda}}(\mathfrak{X})\to H^0((\pone)^M, (\tensor_{a=1}^M p_a^* \Omega^1_{\pone})(D))$ given by  $\langle\Psi|\mapsto \Omega_{\beta}(\langle \Psi|)$, is injective if $M=\sum n_p$.
\end{lemma}
\begin{proof}
Assume $\Omega=0$. We can successively take the residues of the form $\Omega$ in any of the variables. For example if we take the residue of $\Omega$ about $t_1=z_1$, we get a form
$$\langle\Psi| f_{\beta(2)}(t_2)\tensor\dots\tensor f_{\beta(M)}(t_M)\mid (f_{\beta(1)}|\lambda_1\rangle)\tensor |\lambda_2\rangle\tensor\dots\tensor|\lambda_N\rangle.$$

For every partition of $[M]$ into $N$ subsets $I_1\cup\dots\cup I_N$ and every enumeration of elements in $I_j$
$$I_j=\{i_j(1),\dots, i_j(k_j)\},$$ we learn that
$$\langle\Psi| \prod_{j=1}^N f_{\beta(i_j(1))}f_{\beta(i_j(2))}\dots f_{\beta(i_j(k_j))}|\lambda_j\rangle=0.$$
 Since the image of $\langle\Psi|$ in the dual of the tensor product $(\tensor V_{\lambda_j})^*$ is non-zero and invariant under the action of $\frg$, it follows that $\langle\Psi|=0$: The weight spaces in the representations $V_{\lambda_j}$ are obtained from the highest weight vectors by applying $f_{\alpha}$'s with $\alpha$ simple, and we should only consider values of $\langle\Psi|$ on vectors which are in the $0$-weight space of $\frh$.
\end{proof}
\subsection{Some properties of correlation functions}\label{eroica}
Using the gauge condition, it is possible to ``simplify'' correlation functions of the form $\Omega=\langle\Psi|\prod_{a=1}^{M} f_{\beta(a)}(t_a)|\vec{\lambda}\rangle$, so that a particular variable $t_a$ no longer appears
inside the correlation function: $\Omega$ is $dt_a\tensor$ the quantity
$$\sum_{j=1}^N \frac{1}{t_a-z_j}\langle\Psi|\prod_{b=1,b\neq a}^M f_{\beta(b)}(t_b)| \rho_j (f_{\beta(a)})|\vec{\lambda}\rangle\ +\  \sum_{b=1, b\neq a }^M \frac{1}{t_a-t_b} \langle\Psi|[f_{\beta(a)},f_{\beta(b)}](t_b)\prod_{\ell=1,\ell\neq a, \ell\neq b}^M f_{\beta(\ell)}(t_{\ell})|\vec{\lambda}\rangle$$
We have used the rule
\begin{equation}\label{notatio}
dt_a\tensor dt_1\dots dt_{a-1}\dots dt_{a+1}\dots dt_M= dt_1dt_2\dots dt_{a-1}dt_a\dots dt_M.
 \end{equation}
Note that if $\beta(a) + \beta(b)$ is not a  root then, $[f_{\beta(a)},f_{\beta(b)}]=0$, and if it is a root, then $[f_{\beta(a)},f_{\beta(b)}]$ is equal to a non-zero multiple of  $f_{\beta(a)+\beta(b)}$
(note that we have picked these elements arbitrarily, see Definition ~\ref{gold}).
\begin{remark}
We do not have summands with a polar term of order $2$ in $(t_a-t_b)$. This is because $(f_{\delta_1},f_{\delta_2})=0$ for all positive roots $\delta_1$ and $\delta_2$. Introduction of roots
$e_{\delta}$ will complicate the situation here, and introduce such polar terms (these will correspond to motions along the Harder-Narasimhan strata).
\end{remark}

From the above considerations, it follows that $\Omega$ as a form in $t_a$ has at most  simple pole at each of the $z_i$'s and at the $t_b's$ (if $\beta(a) + \beta(b)$ is not a root then there is no pole at $t_b$).
We can iterate this procedure and obtain an expression for $\Omega$ as a sum of terms, each of which has a
simple denominator of total degree $M$ (which shows that correlation are log forms in the sense of Hodge theory), see Section ~\ref{quartet} (and Proposition ~\ref{aty}) for a more refined statement. The refined statement will be used to compare $\Omega$ with the Schechtman-Varchenko forms.

Some properties of $\Omega$ are not easy to see from such an expression. The theory of ~\cite{TK,TUY,Ueno}  allows us to expand $\Omega$ in a power series (on suitable angular sectors) as collections of the $t$-points come together (see Section ~\ref{wsm} and Proposition ~\ref{richardthomas}).

\subsection{The master function}\label{cottard}
Let $\kappa = k+g^*$ where $g^*$ is the dual Coxeter number of $\frg$. The following master function was
discovered by Schechtman-Varchenko  ~\cite{SV}:
$$\mathcal{R}=\displaystyle\prod_{1\leq i<j\leq N}(z_i-z_j)^{\frac{-(\lambda_i,\lambda_j)}{\kappa}}\displaystyle\prod_{a=1}^M \displaystyle\prod_{j=1}^N(t_a-z_j)^{\frac{(\lambda_j,\beta(a))}{\kappa}}
\displaystyle\prod_{1\leq a< b\leq M}  (t_a-t_b)^{\frac{-(\beta(a),\beta(b))}{\kappa}}$$
We will have occasion to use the master function even when $M\neq \sum n_p$ (in such a case the $\beta(a)$'s will be positive, but not necessarily simple roots). In fact it is convenient to have a definition of the master function even in the case $\beta$ is defined on a subset $A$ of $[M]$ (with $M=\sum n_p$ in our applications).\footnote{The relevant situation arises when points $t_i$'s come together.} In this case the master function is a function of  variables $\{t_a, a\in A\}$ and
$$\mathcal{R}=\displaystyle\prod_{1\leq i<j\leq N}(z_i-z_j)^{\frac{-(\lambda_i,\lambda_j)}{\kappa}}\displaystyle\prod_{a\in A} \displaystyle\prod_{j=1}^N(t_a-z_j)^{\frac{(\lambda_j,\beta(a))}{\kappa}}
\displaystyle\prod_{a,b\in A, a< b}  (t_a-t_b)^{\frac{-(\beta(a),\beta(b))}{\kappa}}$$

\subsection{The extension theorem}
Suppose $M=\sum n_p$ (and hence $\beta:[M]\to R\subseteq \Delta_+$).
Let $$X=\{(t_1,\dots, t_M)\in \Bbb{A}^M: t_a\neq t_b, 1\leq a<b\leq M, t_a \neq z_i, i\in [N], a\in [M]\}.$$

Fix a sufficiently divisible positive integer $C$ so that
$$C(\lambda_i,\lambda_j),\ C(\beta(a),\beta(b)),\ C(\beta(a),\lambda_i)\in \Bbb{Z}, \forall a,b\in [M],\ i,j\in [N],\ a<b,\ i<j.$$
Consider an unramified (possibly disconnected) cover of $X$ given by $Y=\{(t_1,\dots, t_M,y)\mid y^{C\kappa} = P\}$,
where
\begin{equation}\label{madhuridixit}
P=\displaystyle\prod_{1\leq i<j\leq N}(z_i-z_j)^{-C(\lambda_i,\lambda_j)}\displaystyle\prod_{a=1}^M \displaystyle\prod_{j=1}^N(t_a-z_j)^{C(\lambda_j,\beta(a))}
\displaystyle\prod_{1\leq a< b\leq M}  (t_a-t_b)^{-C(\beta(a),\beta(b))}.
\end{equation}

Now fix  $\langle\Psi|\in V^{\dagger}_{\vec{\lambda}}(\mathfrak{X})$ and set $\Omega=\Omega_{\beta}(\langle\Psi|)$. The following extension result holds:
\begin{theorem}\label{tempsperdu}
\begin{enumerate}
\item The multi-valued meromorphic form $\mathcal{R}\Omega$ on $X$ is square integrable.
\item The differential form $p^*(\mathcal{R}\Omega)$ extends to an everywhere regular, single valued, differential form of the top order on any smooth and projective compactification $\overline{Y}\supset Y$.
\end{enumerate}
\end{theorem}

\begin{definition}
Let $Z$ be an $n$-dimensional smooth algebraic variety, and $\Gamma$ a possibly multi-valued $n$-form of the following form: For every $p\in Z$, there is an analytic open subset $U$ of $Z$ containing $p$, such that $\Gamma$ can be expressed as  $\Gamma=f\omega$ where
\begin{enumerate}
\item
$\omega$ is a (single valued) meromorphic form on $U$.
\item Some positive integer power of $f$ is a (single valued) meromorphic function on $U$.
\end{enumerate}
Let $S\subset Z$ be an irreducible subvariety. We will denote the logarithmic degree of $\Gamma$ along $S$ by $d^{S}(\Gamma)$. (See ~\cite{loo,Var} for some  background on this concept). Briefly: Blow up  $Z$ along $S$, and let $E$ be the exceptional divisor. Then, $d^S(\Gamma)-1$ is the order of vanishing of (any branch of) $\Gamma$ along $E$.
\end{definition}
\begin{defi}
Let $\widetilde{D}\subset (\pone)^M$ be the sum of the divisor $D$ (defined before Lemma ~\ref{simpsons}) and the divisors $t_a= \infty$ (for $a=1,\dots, M$).
\end{defi}
Ramadas' strategy ~\cite{TRR}  is to prove this kind of theorem by showing that the logarithmic degree of $\mathcal{R}\Omega$ along any abnormal stratum of the divisor $\widetilde{D}$ is positive (see ~\cite{loo,Var} for the notation that we use here). These abnormal strata are of three kinds:
\begin{enumerate}
\item[(S1)] A certain subset of the $t's$ come together (to an arbitrary moving point). That is
$t_1=t_2=\dots= t_L$ after renumbering (possibly changing $\beta$).
\item[(S2)] A certain subset of the $t's$ come together to one of the $z$'s. That is
$t_1=t_2=\dots= t_L=z_1$ after renumbering (possibly changing $\beta$).
\item[(S3)] A certain subset of the $t's$ come together to $\infty$. That is
$t_1=t_2=\dots= t_L=\infty$ after renumbering (possibly changing $\beta$).
\end{enumerate}
More precisely we prove the following theorem which implies Theorem ~\ref{tempsperdu}.
\begin{theorem}\label{Proust}
Assume that $M=\sum n_p$. The logarithmic degree of $\mathcal{R}\Omega$ along each of the strata $(S1)$, $(S2)$ and $(S3)$ is positive.
\end{theorem}

The proof of Theorem ~\ref{Proust} following the proof of a similar assertion in ~\cite{TRR} will be broken up into three parts corresponding to the strata (S1), (S2) and (S3). The proofs of these  three cases share some common features, chiefly that the degrees  of the correlation functions $\Omega$  are  controlled by power series expansions (see Proposition ~\ref{richardthomas}), and also in that the poles along  partial diagonals are seen as a reflection of the properties of the Lie algebra $\frg$ (for example that the logarithmic degree along any strata of $\Omega$ is $\geq 0$).
These three cases also differ in some important details. Therefore, we have given detailed proofs with some ideas and methods repeated three times. The proof of Theorem ~\ref{Proust}  uses ideas from ~\cite{loo,Var} (in addition to those in ~\cite{TRR}): in particular, the proof there of Ramadas' vanishing theorem where variables are set equal to each other, and the pole analysis at finite parabolic points in the extension theorem.

\section{Proof of Theorem  ~\ref{Proust} on the (S1) stratum:  ``When some of the points  come together''.}\label{rememberme}

Let $S$ be the set $t_1=\dots=t_L$. Let $B$ be the completion of the local ring of $(\pone)^M$ along $S$, at its generic point. Then $B=K_S[[u_2,\dots, u_L]]$ where $K_S$ is the function field of $S$, and $u_i=t_i-t_1$. Clear poles of $\Omega$ at the generic point of $S$ by multiplying by $w= \prod_{1\leq a<b\leq L }(t_a-t_b)$:
$$w\Omega = g d\vec{t},\ g= \sum_{d\geq d_0} g_d(u_2,\dots, u_L).$$

Here $g_d$ is a homogeneous polynomial in the $u$'s of (total) degree $d$ with coefficients in $K_S$, and $d_0$ is the smallest possible degree
(so that $g_{d_0}\neq 0$). Our task is to prove that the logarithmic degree of $\mathcal{R}\Omega$ along $S$, written as $d^S(\mathcal{R}\Omega)$, satisfies the inequality
\begin{equation}\label{inegality}
d^S(\mathcal{R}\Omega)=\bigl(d_0-(L(L-1)/2) +(L-1)\bigr) - \sum_{1\leq a<b\leq L}\frac{(\beta(a),\beta(b))}{\kappa}>0
\end{equation}
(the term in the first bracket is the logarithmic degree of $\Omega$).
\subsection{Reduction to a case in which ``the lowest degree terms'' have no poles as $t_a$ approaches $t_b$ for $1\leq a<b\leq L$}
 We will try to reduce the number of variables.  Let us suppose for example that
\begin{itemize}
\item $g_{d_0}$ is not divisible by $u_2$.
\end{itemize}

If $\beta(1)+\beta(2)$ is not a positive root then $\Omega$ does not have a pole along $t_1=t_2$, so automatically $g_{d_0}$ is divisible by $u_2$.
So let us assume that $\beta(1)+\beta(2)$ is a positive root $\delta$. Now divide $g_d$ by $u_2$ with remainder:
$$g_{d}= r_{d}(u_3,\dots, u_L) + u_2 q_d(u_2,\dots,u_L)$$
and $r_{d_0}\neq 0$.

Let us consider in parallel the new correlation function $\Omega^*$ with variables $t_1,t_3,\dots, t_M$ and $\beta^*(1)=\delta$, along the stratum $S^*$ given by $t_1=t_3=\dots=t_L$. We multiply
the corresponding correlation function by $$w^*=\prod_{1\leq a<b\leq L, a\neq 2, b\neq 2}(t_a-t_b)$$ write an expansion
$$w^*\Omega^*= \sum_{d\geq \tilde{d}_0} \tilde{g}_d(u_3,\dots, u_L) dt_1 dt_3 \dots dt_N$$

We know that(see equation ~\eqref{notatio} and Section ~\ref{eroica}) $$\Omega= dt_2\tensor\frac{\Omega^*}{t_2-t_1} +\hat{\Omega}$$
where $\hat{\Omega}$ is regular at $t_1=t_2$ (and at most poles of the first order as $t_1$ approaches the other variables).
 Multiply by $w$, and get
$$w\Omega=dt_2\tensor w^*\prod_{a\geq 3} (t_2-t_a){\Omega^*}  +w\hat{\Omega}$$

Set $u_2=0$ (that is, $t_2=t_1$) now and get $r_d(u_3,\dots, u_L)= -\tilde{g}_d(u_3,\dots,u_L) \prod_{a>2}(t_1-t_a)$. We can therefore conclude that $\tilde{g}_{d_0-(L-2)}\neq 0$. Now if $\tilde{g}_{d_0-(L-2)-s}\neq 0$, then $r_{d_0-s}\neq 0$ which would imply that $g_{d_0-s}\neq 0$. Therefore we conclude
\begin{equation}\label{clef}
d_0 =\tilde{d_0} +(L-2)
\end{equation}
Therefore (where $\mathcal{R}^*$ is the master function of $\beta^*$ and the variables $t_1,t_3,\dots, t_L$)
$$d^S(\beta,\mathcal{R}\Omega)-d^{S^*}(\beta^*,\mathcal{R}^*\Omega^*)=$$
$$(L-2)+ (L-1)(L-2)/2 - L(L-1)/2 +(L-1)-(L-2) - \frac{(\beta(1),\beta(2))}{\kappa}$$
$$= -\frac{(\beta(1),\beta(2))}{\kappa}.$$

We obtain the equalities
\begin{equation}\label{simplerthing}
d^S(\Omega)= d^{S^*}(\Omega^*)
\end{equation}

\begin{equation}\label{G2thing}
d^S(\mathcal{R}\Omega)= d^{S^*}(\mathcal{R}^*\Omega^*) -\frac{(\beta(1),\beta(2))}{\kappa}
\end{equation}

(think of this as ``caused by'' the loss of one variable $t_2$, a polar term $(t_1-t_2)^{-1}$, and a fractional power $(t_1-t_2)^{-\frac{(\beta(1),\beta(2))}{\kappa}}$).

We continue this process as far as possible (we have not used the fact that $\beta(a)$ are simple roots above, the same arguments apply even if $\beta(a)$ are arbitrary positive roots). At every step ``two variables come together''. We will remove the one with the larger subscript from our list of variables. So we will have a $\tilde{\beta}$ defined on a subset $A$ of $[M]$,  a new correlation function $\widetilde{\Omega}$, a new stratum $\widetilde{S}$ (of some points in $\{t_a:a\in A\}$ equalling each other) and a new master function $\widetilde{\mathcal{R}}$ corresponding to $\tilde{\beta}$ and the variables $\{t_a\}, a\in A$. For every
$t\in A$, let the set of $a\in [M]$ that got together to give $t$ be denoted by $\neww(t)$ (``the set of maximal ancestors''). Let the set of points that descend from $1,\dots,L$ be denoted by $T$. For example if the process ends at the first step as above, then $A=\{1,3,\dots,M\}$, $T=\{1,3,\dots,L\}$ and the ancestors of $1$ are $1$ and $2$.

For convenience assume $1\in T$.  Note that $A=T\cup\{L+1,\dots,M\}$. The case $|T|=1$ is separately covered below.
\begin{remark} At this point $\widetilde{\Omega}$ may still have poles along a partial diagonal $t_a= t_b$ where $a, b\in T, a\neq b$. Our assumption is only that a suitable ``lowest degree term'' is pole-free.
\end{remark}
The logarithmic degree of $\widetilde{\Omega}$ is at least $|T|-1$, because the appropriate lowest degree term is divisible by all pairwise differences (the role of $t_1$ in the above argument can be played by any of the variables $t_1,\dots,t_L$):
\begin{equation}\label{longrun}
d^{\widetilde{S}}(\widetilde{\Omega})\geq |T|-1
\end{equation}
Now, because of equation ~\eqref{G2thing},
\begin{equation}\label{jalsagar}
d^S(\mathcal{R}\Omega)= d^{\widetilde{S}}(\widetilde{\mathcal{R}}\widetilde{\Omega}) -\sum_{t\in T}\sum_{a,b\in \neww(t), a<b}\frac{(\beta(a),\beta(b))}{\kappa}
\end{equation}
we also have $d^S(\Omega)=d^{\widetilde{S}}(\widetilde{\Omega})$ because of equation ~\eqref{simplerthing} and hence
\begin{equation}\label{beatles10}
d^S(\mathcal{R}\Omega)= d^{\widetilde{S}}(\widetilde{\Omega}) - \sum_{1\leq a<b\leq L}\frac{(\beta(a),\beta(b))}{\kappa}
\end{equation}
Introduce $\gamma=\sum_{a=1}^L \beta(a)=\sum_{t\in T}\tilde{\beta}(t)$,
\begin{equation}\label{phillips}
(\gamma,\gamma)-\sum_{a=1}^L({\beta}(a),{\beta}(a))= 2 \sum_{1\leq a<b\leq L}({\beta}(a),{\beta}(b))
\end{equation}
so we find
\begin{equation}\label{beatles30}
d^S(\mathcal{R}\Omega)= d^{\widetilde{S}}(\widetilde{\Omega})-\frac{(\gamma,\gamma)}{2\kappa} +
\sum _{a=1}^L \frac{(\beta(a),\beta(a))}{2\kappa}
\end{equation}
We have two estimates for
$d^{\widetilde{S}}(\widetilde{\Omega})$: it is at least $|T|-1$ and also at least  $\frac{(\gamma,\gamma)}{2k}-1$ (by Lemma ~\ref{basic}).  Assume that both estimates lead to lower bounds for $d^S(\mathcal{R}\Omega)$, which are $\leq 0$. The basic idea in the proof is that one gets both a lower bound and an upper bound for $(\gamma,\gamma)$ which are in conflict.

We obtain,
\begin{equation}\label{nooner1}
|T|-1 -\frac{(\gamma,\gamma)}{2\kappa} +\sum _{a=1}^L \frac{(\beta(a),\beta(a))}{2\kappa}\leq 0
\end{equation}
and
$$\frac{(\gamma,\gamma)}{2k}-1-\frac{(\gamma,\gamma)}{2\kappa} + \sum_{a=1}^L \frac{(\beta(a),\beta(a))}{2\kappa}\leq 0$$
that is
$$\frac{g^*(\gamma,\gamma)}{2k\kappa}\leq 1- \sum_{a=1}^L \frac{(\beta(a),\beta(a))}{2\kappa}$$
or that
$$\frac{(\gamma,\gamma)}{2\kappa}\leq \frac{k}{g^*}-\frac{k}{g^*}\sum_{a=1}^L \frac{(\beta(a),\beta(a))}{2\kappa}$$
which gives (using inequality ~\eqref{nooner1}), the inequality
$$|T| \leq 1+ \frac{k}{g^*}-\bigl(\frac{k+g^*}{g^*}\bigr)\sum_{a=1}^L \frac{(\beta(a),\beta(a))}{2\kappa}=\frac{\kappa}{g^*}- \sum_{a=1}^L \frac{(\beta(a),\beta(a))}{2g^*}$$

We will now cover the case $|T|<\frac{\kappa}{g^*}$. Use Lemma ~\ref{ravishankar} and equation ~\eqref{jalsagar}, to see that
$$d^S(\mathcal{R}\Omega)\geq d^{\widetilde{S}}(\widetilde{\mathcal{R}}\widetilde{\Omega})
=|T|-1 -\sum _{t\neq t', t<t'}\frac{(\tilde{\beta}(t),\tilde{\beta}(t'))}{\kappa}$$

Using the numerical inequalities, $(\beta(t),\beta(t'))\leq 2$, we find
$$d^S(\mathcal{R}\Omega)\geq  |T|-1 -\frac{|T|(|T|-1)}{\kappa}$$
 If $|T|<\kappa$ and $|T|>1$ then the above quantity is positive. If $|T|=1$, then Lemma ~\ref{ravishankar} below assures us that $d^S(\mathcal{R}\Omega)> d^{\widetilde{S}}(\widetilde{\mathcal{R}}\widetilde{\Omega})=0.$ We have thus covered
all cases (we are assuming that $L>1$).
\begin{lemma}\label{ravishankar}
For every positive root $\delta=\sum_{i=1}^n\delta_i$ where $\delta_i$ are positive simple roots (possibly repeated in the sum), and $n>1$,
$$I(\delta)=\sum_{1\leq i< j\leq n}(\delta_i,\delta_j)<0.$$
\end{lemma}
\begin{proof}
We divide the proof into two cases:
\begin{enumerate}
\item $\frg\neq G_2$: If the theorem is true for $\delta_1$ and $\delta_2$ then the theorem is true for their sum
$\delta_1+\delta_2$, because
$$I(\delta_1+\delta_2)= I(\delta_1)+I(\delta_2)+ (\delta_1,\delta_2)$$
But since $\delta _1$, $\delta_2$ are positive roots, whose sum is also a root, and $\frg\neq G_2$, we have \footnote{Use results on Page 278 in ~\cite{bourbaki}. The inequality is strict in the ADE case.}
 $(\delta_1,\delta_2)\leq 0$. We can now induct and get the desired statement. At the first step,
$\delta_1$ and $\delta_2$ are distinct simple roots whose sum is a root and $I(\delta_1+\delta_2)=(\delta_1,\delta_2)<0$ (a strict inequality, see ~\cite{FH}, Lecture 21).
\item $\frg=G_2$. This situation is easy enough for a direct verification. There are $6$ positive roots: $$\alpha_1,\ \alpha_2,\ \alpha_1+\alpha_2,\ 2\alpha_1+\alpha_2,\ 3\alpha_1+\alpha_2,\ 3\alpha_1+2\alpha_2$$
also note (employing the normalization $(\theta,\theta)=2$) $$(\alpha_1,\alpha_1)=\frac{2}{3}, (\alpha_2,\alpha_2)=2, (\alpha_1,\alpha_2)=-1$$
\end{enumerate}
\end{proof}
\begin{lemma}\label{basic}\footnote{For $\frg=\mathfrak{sl}_2$, one gets Ramadas' vanishing theorem when applied to the conformal block situation.}
$d^{\widetilde{S}}(\widetilde{\Omega})+1\geq \frac{1}{2k}(\gamma,\gamma)$.
\end{lemma}
\begin{proof}
For ease in notation let us assume $\beta=\tilde{\beta}$ and drop the assumption that $\beta(a)$ are simple roots.
We will now expand $\Omega$ by a power-series in $u_2=t_2-t_1,\dots, u_L=t_L-t_1$. To apply the considerations of Section ~\ref{wsm} below, write, (by propagation by vacuum introduce the vacuum representation at $z_0=t_1$ and consider the vector (where $\xi_0=z-t_1$, so that $u_i=\xi_0(t_i)$)
$f_{\beta(1)}(-1)|0\rangle$ at that point):
\begin{equation}\label{summand}
\Omega= \sum_{b_2,\dots, b_L} u_2^{-b_2-1}\dots u_L^{-b_L-1} \Omega_{\vec{b}}
\end{equation}
in the angular sector
$0<|u_L|<\dots<|u_2|<\epsilon$ (with $\epsilon$ depending upon $t_1$)
where $\Omega_{\vec{b}}$ equals $$\langle{\Psi}|f_{\beta(L+1)}(t_{L+1})\dots f_{\beta(M)}(t_M)\rho_0(f_{\beta(2)}(\xi_0^{b_2}))\dots \rho_0(f_{\beta(L)}(\xi_0^{b_L}))|f_{\beta(1)}(-1)|0\rangle\tensor|\vec{\lambda}\rangle dt_1du_2\dots du_L.$$

The logarithmic degree of the summand in \eqref{summand} is $-\sum_{a=2}^L b_a$. Suppose that this summand is non-zero. Let $\gamma=\sum_{a=1}^L \beta(a)$, applying Proposition ~\ref{richardthomas}, we find,
$$-1+\sum_{a=2}^L b_a\leq  -\frac{(\gamma,\gamma)}{2k}$$
hence
$$-\sum_{a=2}^L b_a \geq \frac{(\gamma,\gamma)}{2k}-1$$
Note that
$d^S(\Omega)$,  the logarithmic degree of $\Omega$, is (at least) the minimum of  $-\sum_{a=2}^L b_a$ (such that $\rho_0(f_{\beta(2)}(\xi_0^{b_2}))\dots \rho_0(f_{\beta(L)}(\xi_0^{b_L}))|f_{\beta(1)}(-1)|0\rangle\neq 0$). The desired inequality follows (see Section ~\ref{heisenberg}).
\end{proof}

\section{Proof of Theorem ~\ref{Proust} for the (S2) stratum : ``When some points come together at a finite parabolic point''.}

Let $S$ be the set $t_1=\dots=t_L=z_1$. Let $B$ be the completion of the local ring of $(\pone)^M$ along $S$, at its generic point. Then $B=K_S[[u_1,u_2,\dots, u_L]]$ where $K_S$ is the function field of $S$, and $u_i=t_i-z_1$. Clear poles of $\Omega$ at the generic point of $S$ by multiplying:
$$w= \prod_{1\leq a<b\leq L}(t_a-t_b)\prod_{a=1}^L (t_a-z_1)$$
and then
$$w\Omega = g d\vec{t},\ g= \sum_{d\geq d_0} g_d(u_1,\dots, u_L) $$

where $g_d$ is a homogeneous polynomial in the $u$'s of (total) degree $d$ with coefficients in $K_S$, and $d_0$ is the smallest possible degree
(so that $g_{d_0}\neq 0$). Our task is to prove that
$$d^S(\mathcal{R}\Omega)=\bigl(d_0-(L(L-1)/2 + L) +L \bigr) - \sum_{1\leq a<b\leq L}\frac{(\beta(a),\beta(b))}{\kappa} + \sum_{a=1}^L\frac{(\lambda_1,\beta(a))}{\kappa}>0$$
which simplifies to
\begin{equation}\label{inegality1}
d^S(\mathcal{R}\Omega)=\bigl(d_0-L(L-1)/2 \bigr) - \sum_{1\leq a<b\leq L}\frac{(\beta(a),\beta(b))}{\kappa} + \sum_{a=1}^L\frac{(\lambda_1,\beta(a))}{\kappa}>0
\end{equation}
\subsection{Reduction to a case in which ``the lowest degree terms'' have no poles as $t_a$ approaches $t_b$ for $1\leq a<b\leq L$}
 We will reduce (by induction) to the case (but no longer requiring that $\beta(a)$ are simple roots): $g_{d_0}$ is divisible by $(t_1-t_2)$, so that $d_0\geq L(L-1)/2$.  Let us suppose for example that $g_{d_0}$ is not divisible by $(t_1-t_2)$. If $\beta(1)+\beta(2)$ is not a positive root then $\Omega$ does not have a pole along $t_1=t_2$, so automatically $g_{d_0}$ is divisible by $(t_1-t_2)$.

So let us assume that $\beta(1)+\beta(2)$ is a positive root $\delta$.
Now divide $g_d$ by $t_1-t_2=u_1-u_2$ with remainder:
$$g_{d}= r_{d}(u_1,u_3,\dots, u_L) + (u_1-u_2) q_d(u_1,\dots,u_L)$$
and $r_{d_0}\neq 0$.

Let us consider in parallel the new correlation function $\Omega^*$ with variables $t_1,t_3,\dots, t_M$ and $\beta^*(1)=\delta$, along the stratum $S^*$ given by $t_1=t_3=\dots=t_L=z_1$. We multiply
the corresponding correlation function by $$w^*=\prod_{1\leq a<b\leq L, a\neq 2, b\neq 2}(t_a-t_b) \prod_{a=1, a\neq 2}^L (t_a-z_1)$$ write an expansion
$$w^*\Omega^*= \sum_{d\geq \tilde{d}_0} \tilde{g}_d(u_1, u_3,\dots, u_L) dt_1 dt_3 \dots dt_N$$

We know that $$\Omega= dt_2\tensor\frac{\Omega^*}{t_2-t_1} +\hat{\Omega}$$
where $\hat{\Omega}$ is regular at $t_1=t_2$.
 Multiply by $w$, and get
$$w\Omega=dt_2\tensor w^*(t_2-z_1)\prod_{a\geq 3} (t_2-t_a)\Omega^*  +w\hat{\Omega}$$

Note that $w\hat{\Omega}$ does not have poles along the partial diagonals $t_a=t_b$ where $1\leq a< b\leq L$, and equal to $0$ when $u_2=u_1$. Set $u_2=u_1$ (that is, $t_2=t_1$) now and get $r_d(u_1,u_3,\dots, u_L)= (t_1-z_1)\prod_{a>2}(t_1-t_a)\tilde{g}_d(u_1,u_3,\dots,u_L)$. We can therefore conclude that $\tilde{g}_{d_0-(L-1)}\neq 0$. Now if $\tilde{g}_{d_0-(L-1)-s}\neq 0$, then $r_{d_0-s}\neq 0$ which would imply that $g_{d_0-s}\neq 0$. Therefore we conclude

\begin{equation}\label{clef1}
d_0 =\tilde{d_0} +(L-1)
\end{equation}

By a calculation similar to that of equation ~\eqref{G2thing},
we obtain the equality (here $\mathcal{R}^*$ is the new master function corresponding to $\beta^*$ and the variables $t_1,t_3,\dots, t_L$)
\begin{equation}\label{G2thingneq}
d^S(\mathcal{R}\Omega)= d^{S^*}(\mathcal{R}^*\Omega^*) -\frac{(\beta(1),\beta(2))}{\kappa}
\end{equation}
(think of this as ``caused by'' the loss of one variable $t_2$, a polar term $(t_1-t_2)^{-1}$, and a fractional power $(t_1-t_2)^{-\frac{(\beta(1),\beta(2))}{\kappa}}$).
As before, we also have the equality $d^S(\Omega)= d^{S^*}(\Omega^*)$.

We continue this process until we obtain a situation where we cannot reduce any more. So we will have  $\tilde{\beta},T\subseteq A, \widetilde{\Omega}, \widetilde{S}, \widetilde{\mathcal{R}},$ etc, as in Section ~\ref{rememberme} (where $\widetilde{S}$ corresponds to  points $\{t_a:a\in T\}$ equalling $z_1$). Let $\gamma=\sum_{a=1}^L \beta(a)=\sum_{t\in T}\tilde{\beta}(t)$ (as before) and obtain
\begin{equation}\label{beatles300}
d^S(\mathcal{R}\Omega)= d^{\widetilde{S}}(\widetilde{\Omega})+\frac{2(\lambda_1,\gamma)-(\gamma,\gamma)}{2\kappa} +\sum_{a=1}^L \frac{(\beta(a),\beta(a))}{2\kappa}
\end{equation}

We have two estimates for
$d^{\widetilde{S}}(\widetilde{\Omega})$: it is at least $0$ (because there no poles in the ``smallest degree term'' as $t$'s approach each other, but there may be a first order pole in the ``lowest degree term'' as the $t$'s approach $z_1$) and also at least by Lemma ~\ref{basicnew}, $\frac{(\gamma,\gamma)-2(\lambda_1,\gamma)}{2k}$. Assuming that both estimates lead to lower bounds for $d^S(\mathcal{R}\Omega)$, which are $\leq 0$, we conclude that

\begin{equation}\label{nooner10}
-\frac{(\gamma,\gamma)-2(\lambda_1,\gamma)}{2\kappa} +\sum_{a=1}^L \frac{(\beta(a),\beta(a))}{2\kappa}\leq 0
\end{equation}
and
$$\frac{(\gamma,\gamma)-2(\lambda_1,\gamma)}{2k}-\frac{(\gamma,\gamma)-2(\lambda_1,\gamma)}{2\kappa} +\sum_{a=1}^L\frac{(\beta(a),\beta(a))}{2\kappa}\leq 0$$
that is
\begin{equation}\label{harmony}
\frac{g^*\bigl((\gamma,\gamma)-2(\lambda_1,\gamma)\bigr)}{2k\kappa}\leq - \sum_{a=1}^L \frac{(\beta(a),\beta(a))}{2\kappa}
\end{equation}
Inequality ~\eqref{nooner10} implies that the quantity $(\gamma,\gamma)-2(\lambda_1,\gamma)$ is $>0$, while
inequality ~\eqref{harmony} implies that it is $<0$, a contradiction.

\begin{lemma}\label{basicnew}
$d^{\widetilde{S}}(\widetilde{\Omega})\geq \frac{1}{2k}\bigl((\gamma,\gamma)-2(\lambda_1,\gamma)\bigr)$.
\end{lemma}
\begin{proof}
For ease in notation let us assume that $\beta=\tilde{\beta}$ and drop the assumption that $\beta(a)$ are simple roots. The proof parallels that of Lemma ~\ref{basic}.
 Expand $\Omega=\widetilde{\Omega}$ by a power-series in $u_1=t_1-z_1,\dots, u_L= t_L-z_1$ (on a suitable angular sector). To apply the considerations of Section ~\ref{wsm} below, write
\begin{equation}\label{summand1}
\Omega= \sum_{b_1,\dots, b_L} u_1^{-b_1-1}\dots u_L^{-b_L-1} \Omega_{\vec{b}}
\end{equation}
where $\Omega_{\vec{b}}$ equals $$\langle{\Psi}|f_{\beta(L+1)}(t_{L+1})\dots f_{\beta(M)}(t_M)\rho_1(f_{\beta(1)}(b_1)\dots \rho_1(f_{\beta(L)}(b_L))|\vec{\lambda_1}\rangle dt_1dt_2\dots dt_L.$$

The logarithmic degree of the summand (assumed to be non-vanishing) in \eqref{summand1} is $-\sum_{a=1}^L b_a$. Let $\gamma=\sum_{a=1}^L \beta(a)$, applying Proposition ~\ref{richardthomas}, we find (where as above, $d^S(\Omega)$ is the logarithmic degree of $\Omega$), then
$$\sum_{a=1}^L  b_a\leq  \frac{2(\lambda_1,\gamma)-(\gamma,\gamma)}{2k} $$
and hence,
$$d^S(\Omega)\geq   \frac{(\gamma,\gamma)-2(\lambda_1,\gamma)}{2k}.$$
\end{proof}
\section{Proof of Theorem ~\ref{Proust} for the (S3) stratum: ``When some points come together at  infinity''.}\label{deligne}

Let $S$ be the set $t_1=\dots=t_L=\infty$. Let $B$ be the completion of the local ring of $(\pone)^M$ along $S$, at its generic point. Then $B=K_S[[u_1,u_2,\dots, u_L]]$ where $K_S$ is the function field of $S$, and $u_i=\frac{1}{t_i}$.
We will use the coordinate $u=\frac{1}{z}$ at infinity
 Clear poles of $\Omega$ at the generic point of $S$ by multiplying:
$$\Omega'= w \Omega,\  w= \prod_{1\leq a<b\leq L}(u_a-u_b)$$
and then
$$\Omega' = g d\vec{u},\ g= \sum_{d\geq d_0} g_d(u_1,\dots, u_L) $$

where $g_d$ is a homogeneous polynomial in the $u$'s of (total) degree $d$ with coefficients in $K_S$, and $d_0$ is the smallest possible degree
(so that $g_{d_0}\neq 0$). Note that the form $\Omega$ is holomorphic in each of its variables at infinity.
Our task is to prove
\begin{equation}\label{inegality11}
d^S(\mathcal{R}\Omega)=\bigl(d_0-L(L-1)/2 +L\bigr) - \sum_{1\leq a<b\leq L}\frac{(\beta(a),\beta(b))}{\kappa} - \sum_{a=1}^L\frac{(\beta(a),\beta(a))}{\kappa}>0
\end{equation}
In the above, we have used the fact that the functions $(t_a-t_b)$ and $(t_a-z_i)$  have poles of order $1$ at infinity in $t_a$, so the master function acquires an additional pole at infinity in each variable $t_a$  of order $\frac{1}{\kappa}$ times the quantity
$$\sum_{i=1}^N (\lambda_i,\beta(a))-\sum_{b\neq a} (\beta(a), \beta(b))= (\sum_{i=1}^N  \lambda_i-\sum_{b=1}^M \beta(b),\beta(a)) +
(\beta(a),\beta(a))= (\beta(a),\beta(a)).$$
\subsection{Reduction to a case in which ``the lowest degree terms'' have no poles as $t_a$ approaches $t_b$ for $1\leq a<b\leq L$}
 Let us suppose for example that $g_{d_0}$ is not divisible by $(u_1-u_2)$. If $\beta(1)+\beta(2)$ is not a positive root then $\Omega$ does not have a pole along $t_1=t_2$, so automatically $g_{d_0}$ is divisible by $(u_1-u_2)$.

So let us assume that $\beta(1)+\beta(2)$ is a positive root $\delta$. Now divide $g_d$ by $(u_1-u_2)$ with remainder:
$$g_{d}= r_{d}(u_2,\dots, u_L) + (u_1-u_2) q_d(u_1,\dots,u_L)$$
and $r_{d_0}\neq 0$.

Let us consider in parallel the new correlation function $\Omega^*$ with variables $t_1,t_3,\dots, t_M$ and $\beta^*(1)=\delta$, $\beta^*(a)=\beta(a), a>2$, along the stratum $S^*$ given by $t_1=t_3=\dots=t_L=\infty$. We multiply
the corresponding correlation function by $$w^*=\prod_{1\leq a<b\leq L, a\neq 2, b\neq 2}(u_a-u_b)$$ write an expansion
$$w^*\Omega^*= \sum_{d\geq \tilde{d}_0} \tilde{g}_d(u_1, u_3,\dots, u_L) du_1 du_3 \dots du_N$$

We know that $$\Omega= du_2\tensor\frac{\Omega^*}{u_2-u_1} +\hat{\Omega}$$
where $\hat{\Omega}$ is regular at $u_1=u_2$.
 Multiply by $w$, and get
$$w\Omega=du_2\tensor w^*\prod_{a\geq 3}  (u_1-u_a)\Omega^*  +w\hat{\Omega}$$

Set $u_1-u_2=0$ now and get $r_d(u_1,u_3,\dots, u_L)= \tilde{g}_d(u_1,u_3,\dots,u_L) \prod_{a>2}(u_1-u_a)$. We can therefore conclude that $\tilde{g}_{d_0-(L-2)}\neq 0$. Now if $\tilde{g}_{d_0-(L-2)-s}\neq 0$, then $r_{d_0-s}\neq 0$ which would imply that $g_{d_0-s}\neq 0$. Therefore we conclude

\begin{equation}\label{clef12}
d_0 =\tilde{d_0} +(L-2)
\end{equation}

We calculate,
$$d^S(\mathcal{R}\Omega)-d^{S^*}(\mathcal{R}^*\Omega^*)=(L-2)+ (L-1)(L-2)/2 - L(L-1)/2 + 1 + \frac{(\beta(1),\beta(2))}{\kappa}$$
$$= \frac{(\beta(1),\beta(2))}{\kappa},$$
and
$$d^S(\Omega)=d^{S^*}(\Omega^*)$$

We continue this process until we obtain a situation where we cannot reduce any more. So we will have  $\tilde{\beta},T\subseteq A, \widetilde{\Omega}, \widetilde{S}, \widetilde{\mathcal{R}},$ etc, as in Section ~\ref{rememberme} (where $\widetilde{S}$ corresponds to points   $\{t_a:a\in T\}$ equalling infinity).

\begin{equation}\label{beatles1}
d^S(\mathcal{R}\Omega)= d^{\widetilde{S}}(\widetilde{\Omega}) - \sum_{1\leq a<b\leq L}\frac{(\beta(a),\beta(b))}{\kappa} - \sum_{a=1}^L\frac{(\beta(a),\beta(a))}{\kappa}
\end{equation}
Introduce, as before  $\gamma=\sum_{a=1}^L \beta(a)=\sum_{t\in T}\tilde{\beta}(t)$, and use equation ~\eqref{phillips} and the
equality
$$
\sum_{t\in T}\sum_{a\in \neww(t)}\frac{(\beta(a),\beta(a))}{2\kappa}=\sum _{a=1}^L \frac{(\beta(a),\beta(a))}{2\kappa}
$$
 to obtain the following expression which is better suited for the final problem:
\begin{equation}\label{beatles3}
d^S(\mathcal{R}\Omega)= d^{\widetilde{S}}(\widetilde{\Omega})-\frac{(\gamma,\gamma)}{2\kappa} -\sum_{t\in T}\sum_{a\in \neww(t)}\frac{(\beta(a),\beta(a))}{2\kappa}
\end{equation}
We have two estimates for
$d^{\widetilde{S}}(\widetilde{\Omega})$: it is at least $|T|$ (because the ``lowest'' degree term is divisible by all pairwise differences, compare with ~\eqref{longrun})and also at least $\frac{(\gamma,\gamma)}{2k}$ (this corresponds to $\lambda_1=0$ in Lemma ~\ref{basicnew}). Assuming that both estimates lead to lower bounds for $d^S(\mathcal{R}\Omega)$, which are $\leq 0$, we conclude that

\begin{equation}\label{nooner}
|T| -\frac{(\gamma,\gamma)}{2\kappa} -\sum_{t\in T}\sum_{a\in \neww(t)}\frac{(\beta(a),\beta(a))}{2\kappa}\leq 0
\end{equation}
and
$$\frac{(\gamma,\gamma)}{2k}-\frac{(\gamma,\gamma)}{2\kappa} -\sum_{t\in T}\sum_{a\in \neww(t)}\frac{(\beta(a),\beta(a))}{2\kappa}\leq 0$$
that is
$$\frac{g^*(\gamma,\gamma)}{2k\kappa}\leq \sum_{t\in T}\sum_{a\in \neww(t)}\frac{(\beta(a),\beta(a))}{2\kappa}$$
or that
$$\frac{(\gamma,\gamma)}{2\kappa}\leq \frac{k}{g^*}\sum_{t\in T}\sum_{a\in \neww(t)}\frac{(\beta(a),\beta(a))}{2\kappa}$$

which gives (using inequality ~\eqref{nooner}) the inequality

$$|T| \leq \bigl(\frac{k+g^*}{g^*}\bigr)\sum_{t\in T}\sum_{a\in \neww(t)}\frac{(\beta(a),\beta(a))}{2\kappa}=\sum_{t\in T}\sum_{a\in \neww(t)}\frac{(\beta(a),\beta(a))}{2g^*}$$

It turns out that for every $t\in T$, the summand $\sum_{a\in \neww(t)}\frac{(\beta(a),\beta(a))}{2g^*}$ on the right hand side is $<1$, leading to a contradiction. More precisely,
\begin{lemma}\label{basiclemma}
Let $\delta$ be a positive root and $\delta=\sum_{i=1}^s \delta_i$ where $\delta_i,\ 1=1,\dots,s$ are simple roots (possibly repeated). Then $\sum_{i=1}^s(\delta_i,\delta_i)<2g^*$.
\end{lemma}
\begin{proof}
The statement  reduces to $\delta=\theta$, the highest
root. Looking at the tables in ~\cite{bourbaki}, we can verify  that if $\theta=\sum b_i\delta_i$ where $\delta_i$ are simple roots (without  repetitions), then
$\sum b_i(\delta_i,\delta_i)= (g^*-1)(\theta,\theta)=2(g^*-1)$. (In the ADE case, the Coxeter number, which is  the same as the dual Coxeter number, equals $\sum b_i+1$. Also, in this case $(\alpha,\alpha)=2$ for all roots, so one has a simpler proof.)
\end{proof}
\section{Power series expansions}
\subsection{Some generalities}\label{heisenberg}
Let $f(z_1,\dots,z_n)$ be a meromorphic function defined in a neighborhood of the origin $0$ in $\Bbb{C}^n$. Assume that $\prod_i z_i \prod_{i<j}(z_i-z_j)f$ is holomorphic at the origin.

The multiplicity of $f$ along a partial diagonal $z_1=\dots=z_s=0$
can be calculated as follows.  We can develop $f$ in a power series as follows: $$f=\frac{\prod_i z_i\prod_{i<j}(z_i-z_j)f}{\prod_i z_i \prod_{i<j}(z_i-z_j)}$$ and then develop the terms ($1\leq i<j\leq s$ on $|z_i|<|z_j|$) $$\frac{1}{z_i-z_j}=-\frac{z_j^{-1}}{1-(z_i/z_j)}$$ in  power series in $z_i/z_j$.

Write
$$f(z_1,\dots, z_N)=\sum_{\vec{b}} z^{\vec{b}} g_{\vec{b}}(z_{s+1},\dots, z_N).$$
where
\begin{enumerate}
\item The summation runs through vectors $\vec{b}=(b_1,\dots,b_s)\in \Bbb{Z}^s$.
\item $z^{\vec{b}}=z_1^{b_1}z_2^{b_2} \dots z_s^{b_s}$.
\item $g_{\vec{b}}(z_{s+1},\dots, z_N)$ is meromorphic with poles only along the partial diagonals.
\item The expansion is valid on $0<|z_1|<|z_2|<\dots<|z_s|<\epsilon$, $|z_{j}-z_{j,0}|<\epsilon'$ for $j=s+1,\dots, N$, and some initial point $(z_{s+1,0}, \dots, z_{N,0})$ (not on any partial diagonals) and positive real numbers $\epsilon$ and $\epsilon'$.
\end{enumerate}
Write $|\vec{b}|=b_1+b_2+\dots + b_s$. The following proposition is immediate.
\begin{proposition}
The multiplicity of  $f$ along $z_1=\dots=z_s=0$ is the minimum value of the set $\{|\vec{b}|: g_{\vec{b}}\neq 0\}$.
\end{proposition}

\subsection {Recollections from vertex operator algebra theory and Kac-Moody algebras}\label{wsm}

Conformal field theory  gives a expansion of the correlation function about any point. The expansion of
the  correlation function (with notation as in Section ~\ref{corro})
$$\omega=\langle\Psi|X_1(Q_1)\dots X_M(Q_M)|\Phi\rangle$$
about $Q_1=\dots=Q_L=P_1$ is, using the coordinate $\xi_1$ (the answer is an element
in $\Bbb{C}((u_1))(((u_{2}))\dots((u_L))$ where $u_i=\xi_1(Q_i)$)  an expression,
valid in the region $0<|u_L|<|u_{L-1}|<\dots<|u_1|<\epsilon$ for some $\epsilon$ (and other variables staying close to initial points), of the form
\begin{proposition}
$$\Omega= \sum_{b_1,\dots, b_L}  \omega_{\vec{b}}u_1^{-b_1-1}\dots u_L^{-b_L-1}$$
with
$$\omega_{\vec{b}}=\langle{\Psi}|X_{L+1}(Q_{L+1})\dots X_M(Q_M)|\rho_1(X_{1}(\xi_1^{b_1}))\dots \rho_1(X_L(\xi_1^{b_L}))|\Phi\rangle d\vec{u}$$
where $d\vec{u}=du_1\dots du_L$.
\end{proposition}
\begin{proof}
We know that multiplying $\Omega$ by a (suitable power of)
$\prod_{i=1}^L u_i\prod_{1\leq i<j\leq L}(u_i-u_j),$ one gets a form holomorphic in $(u_1,\dots, u_L)$ (in a neighborhood of $0\in \Bbb{C}^L$). So such an expansion of $\Omega$ exists. To find the values of $\omega_{\vec{b}}$, we take iterated residues. More precisely by Theorem 3.24 (3) in ~\cite{Ueno}, we have
$$\Res_{u_L=0} u_L^{s} \langle\Psi|X_1(Q_1)\dots X_L(Q_L)X_{L+1}(Q_{L+1})\dots X_M(Q_M)|\Phi\rangle=$$
$$ \langle\Psi|X_1(Q_1)\dots X_{L-1}(Q_{L-1})X_{L+1}(Q_{L+1})\dots X_M(Q_M)|\rho_1(X_L(\xi_1^s))|\Phi\rangle$$
and we may iterate this procedure.
\end{proof}

\begin{remark}\label{silence}
\begin{enumerate}
\item Note that if $|\Phi\rangle=|\Phi_1\rangle\tensor\dots\tensor |\Phi_N\rangle$, then
$$\rho_1(X_{1}(\xi_1^{b_1}))\dots \rho_1(X_L(\xi_1^{b_L}))|\Phi\rangle=(X_1(b_1)\dots X_L(b_L)|\Phi_1\rangle)\tensor|\Phi_2\rangle\dots\tensor |\Phi_N\rangle.$$
We may hence rewrite the expression for $\omega$ as
$$\sum_{b_1,\dots, b_L}  \langle{\Psi}|X_{L+1}(Q_{L+1})\dots X_M(Q_M)|(X_{1}(b_1)\dots X_L(b_L)|\Phi_1\rangle)\tensor|\Phi_2\rangle\tensor \dots |\Phi_N\rangle u_1^{-b_1-1}\dots u_L^{-b_L-1} d\vec{u}$$
\item The above expansion is valid for any choice of $|\Phi_i\rangle\in \mathcal{H}_{\lambda_i}$ (not necessarily in $V_{\lambda_i}$). In one of its applications (Lemma ~\ref{basic}) we will
propagate vacua and let $P_1$ be the new point (after renaming the parabolic points), and consider expansions with $|\Phi_1\rangle= f_{\delta}(-1)|0\rangle$ for a suitable positive root $\delta$.
\end{enumerate}
\end{remark}

\subsection{Vanishing criteria} In a finite dimensional irreducible representation $V_{\lambda}$ of a simple Lie algebra $\frg$ with highest weight $\lambda$, one knows that for any  weight $\lambda'$ occurring in $V_{\lambda}$, $(\lambda',\lambda')\leq (\lambda,\lambda)$. There is an analogous fact for representations of Kac-Moody algebras (see ~\cite{K}, Chapters 7 and 12). Using these results, we now formulate conditions under which an expression of the form
$X_1(b_1)\dots X_L(b_L)|\lambda\rangle$, where $|\lambda\rangle$ is a highest weight vector in a integrable highest weight representation $\mh_{\lambda}$ of $\hat{\frg}$ of level $k$, is necessarily zero.
\begin{proposition}\label{richardthomas}
Suppose that $X_a =f_{\beta(a)} \in \frg_{-\beta(a)}$, where $\beta(a)$ are positive roots (not necessarily simple) and set $\gamma=\sum_{a=1}^L\beta(a)$.
If $X_1(b_1)\dots X_L(b_L)|\lambda\rangle\neq 0$ then
\begin{equation}\label{prakash}
\sum_{a=1}^L b_a\leq \frac{2(\lambda,\gamma)-(\gamma,\gamma)}{2k}.
\end{equation}
\end{proposition}

\begin{remark}
\begin{enumerate}
\item In  ~\eqref{prakash}, the level $k$ appears in the denominator (and not $\kappa=k+g^*$).
In the master functions that we consider, we have $\kappa$ as a fractional exponent. The ``small'' difference between $\kappa$ and $k$ can be deemed responsible in part for the extension theorem (Theorem ~\ref{Proust}). In fact, the difference $g^*$ is in a sense optimal (Lemma ~\ref{basiclemma}).
\item In ~\eqref{prakash1} and ~\eqref{prakash}, inner products  are computed in $\frh^*$.
\end{enumerate}
\end{remark}
\subsection{Proof of Proposition ~\ref{richardthomas}}
Proposition ~\ref{richardthomas} is a direct corollary of Theorem 12.5,  part (d) of ~\cite{K}. To adjust to the notation there, we will include a few brief remarks. Extend the Lie algebra $\hat{\frg}$ (as in ~\cite{K}) by
$$\hat{\frg}'=\ \hat{\frg}+\Bbb{C}d,\ \hat{\frh}'=\frh+\Bbb{C}c + \Bbb{C}{d}$$
with the commutations
$$[d, c] = 0, [d, X(n)] = nX(n)$$

Extend a form $\lambda\in \frh^*$ to $\lambda\in \hat{\frh}'^*$ by setting $\langle \lambda, \Bbb{C}c+\Bbb{C}d\rangle=0$ where $\langle\ ,\ \rangle$ is the standard pairing of a vector space and its dual.

We define the elements $\Lambda_0$ and
$\delta$ in the dual $\hat{\frh}'^*$ = $\frh^*+ \Bbb{C}\Lambda_0+\Bbb{C}\delta$
by $\langle \delta,d\rangle=\langle\Lambda_0,c\rangle= 1$, $\langle \delta, \frh +\Bbb{C}c\rangle=\langle \Lambda_0, \frh +\Bbb{C}d\rangle=0$
We extend the form $(\ , )$ to  $\hat{\frh}'^*$ by putting

$$(\frh^*,\Bbb{C}\Lambda_0+\Bbb{C}\delta)=(\delta,\delta)=(\Lambda_0,\Lambda_0)=0,\ (\delta,\Lambda_0)=1$$

Given a highest weight representation of $\hat{\frg}$ of level $k$, we extend it to a representation of $\hat{\frg}'$ by having $d$ act on the highest weight vector by $0$ (actually define $d$ as $-L_0 +\alpha Id$ for a suitable constant $\alpha$ where $L_0$ is obtained from the Sugawara tensor).

The highest weight of $V_{\lambda}$ corresponds to the weight
$\lambda+ k\Lambda_0\in \hat{\frh}'^*$, and  $X_1(b_1)\dots X_L(b_L)|\lambda\rangle$ corresponds to the weight $\lambda-\gamma + (\sum_{a=1}^L b_a)\delta + k\Lambda_0$.
Apply ~\cite{K}, Theorem 12.5,  part (d) to get the following inequality which implies inequality ~\eqref{prakash}:
\begin{equation}\label{prakash1}
(\lambda-\gamma,\lambda-\gamma)+ 2k\sum_{a=1}^L b_a \leq (\lambda,\lambda).
\end{equation}

\section{The map to a Gauss-Manin system}
\subsection{The KZ/WZW connection:}\label{beethoven}
Let us first recall the KZ connection on conformal blocks as expressed in ~\cite{TUY} (see Section 6.2 in ~\cite{Ueno}).
Consider the tautological family of genus $0$ pointed curves over the configuration space of $N$ distinct points on $\Bbb{A}^1$,  $\mathcal{C}=\Bbb{A}^N-\cup_{i\neq j}\{z=(z_1,\dots,z_N)\mid  z_i=z_j\}$. The pointed curve corresponding to $z\in \mathcal{C}$ will be denoted by $\mathfrak{X}(z)$.

Consider a (local) family  $\langle\Psi|= \langle\Psi|(z)\in
V^{\dagger}_{\vec{\lambda}}(\mathfrak{X}(z))$
which we may view as a $((V_{\lambda_1}\tensor\dots \tensor V_{\lambda_N})^*)^{\frg}$ valued function defined on
a (small) analytic open subset $U$ of $\mathcal{C}$. Such a family is flat if and only if,
for any (constant, i.e independent of $z_1,\dots,z_N$) element
$|\vec{\nu}\rangle=|{\nu}_1\rangle\tensor\dots\tensor
|{\nu}_N\rangle \in V_{\lambda_1}\tensor V_{\lambda_2}\tensor\dots\tensor V_{\lambda_N}$
and  for each $1\leq i\leq n$,
\begin{equation}\label{KZ}
\frac{d}{dz_i}\langle\Psi|\vec{v}\rangle=\frac{1}{\kappa}\sum_{j=1,j\neq i}^N \frac{\langle\Psi|\Omega_{ij}|\vec{\nu}\rangle}{z_i-z_j}
\end{equation}
where $$\Omega_{ij}=\sum_{a=1}^{\dim \frg} \rho_i(J^a)\rho_j(J^a)$$ for any orthonormal basis $\{J^a\}$ of $\frg$.

\subsection{The map to cohomology}
Now consider and fix a  map $\beta:[M]\to R$ (as before, where $R$ the set of simple roots),
with $\mu=\sum \lambda_i=\sum_{a=1}^M \beta(a)$ so that $M=\sum n_p$.

For $z\in \mathcal{C}$, let $Y_z$ be the  cover of
$$X_z=\{(t_1,\dots, t_M)\in \Bbb{A}^M: t_a\neq t_b, 1\leq a<b\leq M, t_a \neq z_i, i=1,\dots,N, a=1,\dots, M\}$$
given by $Y_z=\{(t_1,\dots, t_M,y)\mid y^{C\kappa} = P\}$,
where $P$ was defined earlier (see equation ~\eqref{madhuridixit}).
The spaces $Y_z$ organize into a smooth  family of affine varieties $\mathcal{Y}\to \mathcal{C}$.
\begin{definition}
Define a map $$S:V^{\dagger}_{\vec{\lambda}}(\mathfrak{X}(z))\to H^0(Y_z, \Omega^M_{Y_z}),$$ as follows:
The map $S$ takes $\langle\Psi|(z)\in
V^{\dagger}_{\vec{\lambda}}(\mathfrak{X}(z))$ to the differential $\mathcal{R}\Omega_{\beta}(\langle\Psi|)(z)$ (note that $y=\mathcal{R}$ on $Y_z$):
$$S(\langle\Psi|)=\mathcal{R}\Omega_{\beta}(\langle\Psi|(z)). $$
We recall that $\Omega_{\beta}(\langle\Psi|(z))$ was constructed in Section ~\ref{marcel},and the master function $\mathcal{R}$ was defined in Section ~\ref{cottard}.
\end{definition}

\begin{proposition}\label{ivorytower}
Compose $S$ with  the evident morphism $H^0(Y_z,\Omega^M_{Y_z})\to H^{M}(Y_z, \Bbb{C})$.
The resulting map
$$T:V^{\dagger}_{\vec{\lambda}}(\mathfrak{X}(z))\to H^{M}(Y_z, \Bbb{C})$$
is a flat map (that is, preserves connections) as $z$ varies (with the KZ/WZW connection on the left, and the Gauss-Manin connection on the right hand side).
\end{proposition}
Proposition ~\ref{ivorytower} follows from
\begin{proposition}\label{member}
For every (local) section $$\langle\Psi|(z)\in
V^{\dagger}_{\vec{\lambda}}(\mathfrak{X}(z)),$$ there exist a relative algebraic $M-1$  form $\omega$  on the fibers $\mathcal{Y}\to \mathcal{C}$ (locally over $\mathcal{C}$), such that for $1\leq i\leq N$,
$$ \mathcal{R}\Omega_{\beta}\bigl(\nabla_{\frac{d}{dz_i}}(\langle\Psi|)\bigr)= \frac{d}{dz_i}\mathcal{R}\Omega_{\beta}(\langle\Psi|) + d(\omega)$$
where the differential operator $d$ is the relative $d$-operator for the map
 $\mathcal{Y}\to \mathcal{C}.$
\end{proposition}
\subsection{Proof of Proposition ~\ref{member}}
The key point is that the form $\mathcal{R}\Omega_{\beta}(\langle\Psi|)$ coincides in a suitable sense with a map in the work ~\cite{SV} for which a key flatness statement is proved there. Once this connection is made, the proposition follows immediately from results in ~\cite{SV}.

Let\footnote{We do this to connect with the notation of ~\cite{SV}.} $\lambda=0$. Following the notation of ~\cite{SV} (where any complex value for $\kappa$ is permitted), the element $\langle\Psi|(z)$ produces an element in $(\mathcal{M}^*)_{\lambda}$ ($\mm$ is a tensor product of Verma modules, which surjects upon the tensor product $\tensor_{i=1}^N V_{\lambda_i}$)

A map is defined in ~\cite{SV}:
$$\eta_{\kappa}=\mathcal{M}_{\lambda}^*\to \Omega^{N}(X_z,\mathcal{L}_{\lambda,\kappa}),$$
where $\mathcal{L}_{\lambda,\kappa}$ is a rank one local system. We may take $\mathcal{L}_{\lambda,\kappa}$ to be the trivial vector bundle with the connection $d-d(\log \mathcal{R}$). Locally, in the complex analytic topology we have a flat map
$\mathcal{L}_{\lambda,\kappa}\to (\mathcal{O},\nabla=d)$ which takes  $1$ to $\mathcal{R}^{-1}$.

It is immediate from Proposition ~\ref{aty} and the definitions of the maps $\eta_{\kappa}$ that (with our value for $\kappa=k+g^*$)
\begin{equation}\label{interpretation}
\eta_{-\kappa}(\langle\Psi|)=\mathcal{R}\Omega_{\beta}(\langle\Psi|).
\end{equation}
\begin{remark}\label{pbs}
Write $\langle\Psi|$ in terms of  the basis $\{\delta(\gamma,\epsilon)\}$ in ~\cite{SV}:
 $$\langle\Psi|=\sum_{\gamma,\epsilon}c_{\gamma,\epsilon}\delta(\gamma,\epsilon)$$
where $c_{\gamma,\epsilon}=\langle\Psi|(f(\gamma,\epsilon))$. The formula given for $\eta(\delta(\gamma,\epsilon))$ should be compared with the coefficients in Proposition ~\ref{aty}.
\end{remark}
The flatness statement now follows from Theorems 7.2.5 and 7.2.5' of ~\cite{SV} (compare with Equation (B4) and (B.5) in ~\cite{ATY}). Note that the KZ equation in ~\cite{ATY} differs by a sign from ours.

\subsection{A formula for $\Omega$}\label{quartet}
Suppose $M=\sum n_p$, $\beta:[M]\to R$, $\langle\Psi|\in V^{\dagger}_{\vec{\lambda}}(\mathfrak{X})$ and  $\Omega=\Omega_{\beta}(\langle\Psi|)$. The following proposition gives a formula for $\Omega$
 (see ~\cite{ATY}, Proposition 3.2 for a similar statement, also see (B.4) and (B.5) in ~\cite{ATY}).
\begin{proposition}\label{aty}
$\Omega= \langle\Psi|w(t,z)\rangle\ dt_1\dots dt_M$, where $|w(t,z)\rangle\in V_{\lambda_1}\tensor\dots\tensor V_{\lambda_N}$ is given by the formula
\begin{equation}\label{madhav}
|w(t,z)\rangle=\sum_{part} \prod_{i=1}^N \langle\langle \prod_{a\in I_i} f_{\beta(a)}(t_a) |\lambda_i \rangle\rangle
\end{equation}
with
$$\langle\langle f_{\gamma_1}(u_1) f_{\gamma_2}(u_2)\dots f_{\gamma_q}(u_q) |\lambda_i\rangle\rangle=\sum_{perm}\frac{1}{(u_1-u_2)(u_2-u_3)\dots (u_q-z_i)}(f_{\gamma_1}\dots f_{\gamma_q} |\nu_i\rangle)\in V_{\lambda_i}$$
and where $\sum_{part}$ stands for the summation over all partitions of $I=\{1,\dots,M\}$ into $N$ disjoint parts
$I=I_1\cup I_2\cup\dots\cup I_N$ and $\sum_{perm}$ the summation over all permutations of the elements of $\{1,\dots,q\}$.
 \end{proposition}

\begin{proof}
We will not assume  $M=\sum n_p$ (hence $\beta(a)$ may be arbitrary positive roots), and proceed by induction. Fix vectors $|\nu_i\rangle\in V_{\lambda_i}$, not necessarily  highest weight vectors, for $i=1,\dots, N$. Let $|\vec{\nu}\rangle=|{\nu}_1\rangle\tensor\dots\tensor
|{\nu}_N\rangle$ and consider $\Theta=\langle\Psi| f_{\beta(1)}(t_1)\dots f_{\beta(M)}(t_M)|\vec{\nu}\rangle$. We will prove by induction on $M$ that
\begin{equation}\label{madhav'}
\Theta=\bigl(\sum_{part} \langle\Psi \mid\prod_{i=1}^N \langle\langle \prod_{a\in I_i} f_{\beta(a)}(t_a) |\nu_i \rangle\rangle\bigr) d\vec{t}
\end{equation}
where $d\vec{t}=dt_1\dots dt_M$.

The proof is by induction on $M$. If $M=1$, then the result is clear: start with
$\Omega=\langle{\Psi}|f_{\beta}(t)|\vec{\nu}\rangle$
Now use the function $\frac{1}{z-t}$ and the gauge condition to write
$$\Theta=\sum_{i=1}^N \frac{1}{t_1-z_i}\langle{\Psi}|\rho_i(f_{\beta})|\vec{\nu}\rangle dt_1$$

For $M>1$, let us write $\Theta=f_{\Theta}(t_1,\dots, t_M) d\vec{t}$. We want to show that $f_{\Theta}$ equals the right hand side of equation ~\eqref{madhav} divided by $d\vec{t}$ (we do this to get rid of the   non-commuting $dt_1,\dots, dt_M$). We will show that both sides of the desired equations are equal as functions of $t_1$. It is easy to see that both sides vanish at infinity. We need to show that they have equal polar parts at every finite point. Therefore, we need to analyze the behavior as
\begin{enumerate}
\item $t_1 $ approaches $z_i$: Let $i=1$ for simplicity. The polar part of $\Theta$ is $\frac{1}{t_1-z_1}f_{\tilde{\Theta}}$ corresponding to a correlation function with variables $t_2,\dots, t_M$ (same $\beta$'s) with $|\nu_1\rangle$ changed to
$f_{\beta(1)} |\nu_1\rangle$. On the right hand side we need to consider only terms which have a fraction $\frac{1}{t_1-z_1}$. A little thought convinces us that
the equality of the polar parts at $t_1=z_1$ follows from induction.
\item $t_1$ approaches $t_a$. In this case the polar part  of $f_{\Theta}$  is $\frac{1}{t_1-t_a}f_{\tilde{\Theta}}$ corresponding to a correlation function  with points $t_2,\dots, t_M$, with $f_{\beta(a)}$ replaced by $[f_{\beta(1)},f_{\beta(a)}]$ (a multiple of $f_{\beta(a)+\beta(1)}$ if $\beta(a)+\beta(1)$ is a root, zero otherwise) . On the other side we should be looking at terms which have a $t_1-t_a$ or $t_a-t_1$. First all partitions considered should have $t_1$ and $t_a$ in the same part. So we are looking at words which have $f_{\beta(1)}f_{\beta(a)}$ or $f_{\beta(a)}f_{\beta(1)}$ as sub words. We use the formula

    $$a(u)b(v) f_u f_v - a(v) b(u) f_v f_u=  a(v)b(v)[f_u,f_v]+ O(u-v). $$
\end{enumerate}
\end{proof}

\begin{remark}
The form (with notation as in Proposition ~\ref{aty}) $\mathcal {R} \langle\Psi|w(t,z)\rangle\ dt_1\dots dt_M$ (=$\mathcal{R}\Omega_{\beta}(\langle\Psi|)$) is called a
Schechtman-Varchenko form (these were introduced in ~\cite{SV}).
\end{remark}

\section{Unitarity}
The remaining arguments for unitarity are exactly as in ~\cite{TRR}. For completeness, we include these details. Let $\overline{Y_z}$ be a smooth compactification of $Y$, which varies algebraically with $z$. That is, using resolution of singularities,
find an embedding of $\mathcal{C}$-varieties $\mathcal{Y}\subseteq \overline{\mathcal{Y}}$ so that $\overline{\mathcal{Y}}$ is smooth and projective over $\mathcal{C}$ and $\overline{\mathcal{Y}}-{\mathcal{Y}}$ is a
(relative) divisor with normal crossings.
 \footnote{We may have to shrink $\mathcal{C}$, but a flat unitary structure over a non-empty Zariski open subset of $\mathcal{C}$ yields one over all of $\mathcal{C}$.}

We therefore have a flat map  $$T:V^{\dagger}_{\vec{\lambda}}(\mathfrak{X}(z))\to H^{M}(Y_z, \Bbb{C})$$
which (because of Theorem ~\ref{Proust}) factors through an injective map (see Lemma ~\ref{simpsons}) $$\overline{T}:V^{\dagger}_{\vec{\lambda}}(\mathfrak{X}(z))\to H^{0}(\overline{Y_z}, \Omega^{M})\subseteq H^M(\overline{Y_z}, \Bbb{C})$$
(The second inclusion follows from Hodge theory.)
It follows that
\begin{proposition}\label{consequence}
\begin{enumerate}
\item $T: V^{\dagger}_{\vec{\lambda}}(\mathfrak{X}(z))\to H^{M}(Y_z, \Bbb{C})$ is injective (and flat for connections by ~\cite{SV}, as stated in Proposition ~\ref{ivorytower}).
\item $\overline{T}:V^{\dagger}_{\vec{\lambda}}(\mathfrak{X}(z))\to  H^M(\overline{Y_z}, \Bbb{C})$ is flat for the connections (KZ on one side and Gauss-Manin on the other).
\end{enumerate}
\end{proposition}
The first part uses results of Deligne ~\cite{D}; and the second uses the first, and a semi-simplicity theorem in (pure) Hodge theory. The first part is a theorem of Varchenko for $\frg=\mathfrak{sl}_2$ (see ~\cite{V}, Theorem 14.6.4).

The Hodge form on  $H^M(\overline{Y_z}, \Bbb{C})$ is given on de Rham classes  by $|\omega|^2=(\sqrt{-1})^{M}\int_{\overline{Y_z}} \omega\wedge\bar{\omega}\ $. The restriction of
Hodge form to $H^{0}(\overline{Y_z}, \Omega^{M})\subset H^M(\overline{Y_z}, \Bbb{C})$ gives  a unitary metric.

It is known that the Hodge form  is induced by the cup product on cohomology, and is therefore preserved by the Gauss-Manin connection (note that this connection, in general may not preserve $H^{0}(\overline{Y_z}, \Omega^{M})$).  It then follows that we can induce the desired unitary metric on $V^{\dagger}_{\vec{\lambda}}(\mathfrak{X}(z))$ through the map $\overline{T}$. Theorem ~\ref{introtheorem} therefore holds.
\begin{remark} As in ~\cite{TRR}, the image of $\overline{T}$, lands inside a weight space (for a suitable character) of $H^{0}(\overline{Y_z}, \Omega^{M})$. It would be interesting to obtain a characterization
of the image of $\overline{T}$ (compare with Theorem 7 in ~\cite{looconfa} for the larger KZ system).
\end{remark}
\subsection{An explicit form of the metric}
Given a conformal block
$\langle\Psi|$, let $\pi(\langle\Psi|)=\mathcal{R}\Omega(\langle\Psi|)$ be the corresponding Schechtman-Varchenko form. Then, the unitary metric (upto a constant $C\kappa$) is given by the convergent  integral (as conjectured in ~\cite{FGK}):
$$|\langle\Psi\|^2=(\sqrt{-1})^{M}\int_{(\pone)^M}\pi(\langle\Psi|)\wedge\overline{\pi(\langle\Psi|)}$$
\section{Configuration spaces and moduli spaces} It is easy to see that  $\mathfrak{M}=\mathfrak{M}_{0,N}$  is a smooth affine variety by identifying $\mathfrak{M}$ with the configuration space of $N-3$ points on $\Bbb{A}^1$.
It is known that there is a bundle of conformal blocks $\mv$ on  $\mathfrak{M}$ which comes equipped with a flat connection $\nabla_{\mathfrak{M}}$ induced from the connection on conformal block bundles on configuration spaces (from Section ~\ref{beethoven}).
\begin{remark}
\begin{enumerate}
\item Define two unitary  metrics  on a complex vector space $V$ to be projectively equivalent if they   are positive real multiples of each other. A projective metric on $V$ is an equivalence class under this
 relation.
\item Fix an actual unitary metric on $V$. There is a bijection between elements of   $\operatorname{PGL}(V)$ that carry this metric on $V$ to a metric projectively equivalent to it, and elements of $\operatorname{PU}(V)$.
Here $\operatorname{PGL}(V)=\operatorname{GL}(V)/\Bbb{C}^*$ and $\operatorname{PU}(V)=\operatorname{U}(V)/\operatorname{U}(1)$, and $\operatorname{U}(1)=\{z\in \Bbb{C}\mid z\bar{z}=1\}$.
\end{enumerate}
\end{remark}

Consider the configuration space $\mathcal{C}$  as in Section ~\ref{beethoven}. There is a natural map $\pi:\mathcal{C}\to \mathfrak{M}$. The bundle of conformal blocks over $\mathcal{C}$ considered in this paper is canonically equal to $\pi^* \mv$. We have constructed a unitary metric $(\ ,\ )$ on $\pi^*\mv$ which is preserved by the connection $\nabla_{\mathcal{C}}$ (as in Section ~\ref{beethoven}).

The connection $\nabla_{\mathcal{C}}$ on $\pi^*\mv$  is equal to $\pi^*\nabla_{\mathfrak{M}}$ only as a projective connection\footnote{Recall that the theory in ~\cite{TUY}  produces a well defined flat projective connection.}. Therefore, locally on $\mathcal{C}$, $\nabla_{\mathcal{C}}-\pi^*\nabla_{\mathfrak{M}}
=df\tensor \Id_{\pi^*\mv}$ for a local function $f$ on $\mc$. If $v$ is a local section of $\mv$ on $\mathcal{M}$ and $X$ a tangent vector field on the fibers of $\mathcal{C}\to \mathcal{M}$, we find
 $\nabla_{\mc,X} \pi^*v =(\pi^*\nabla_{\mathfrak{M}})_X v +X(f) v= X(f)v$ and so,
$$X(\pi^*v,\pi^*v)= (\nabla_{\mc,X} \pi^*v,\pi^*v) + (\pi^*v,\nabla_{\mc,X} \pi^*v) = (X(f)+\overline{X(f)}) (\pi^*v,\pi^*v)$$

The above argument shows that for $x,y\in \mc$ with $\pi(x)=\pi(y)=p$, the two metrics on  $\pi^*\mv_x= \pi^*\mv_y=\mv_{p}$ are projectively equivalent. Hence the projective monodromy group of $(\mv,\nabla)$ on $\mathfrak{M}$ about a base point $p\in\mathfrak{M}$ is contained in $\operatorname{PU}(\mv_p)$ where the metric on $\mv_p$ is any element in the projective equivalence class constructed above on $\mv_p$.

\bibliographystyle{plain}
\def\noopsort#1{}

\end{document}